\documentclass[11pt,final]{article}
\usepackage{amsmath,amsthm,amsfonts,amssymb,graphicx,hyperref,enumerate,psfrag}
\usepackage{mathrsfs}
\usepackage{fullpage}

\usepackage[utf8]{inputenc} 

\newtheorem{lemma}{Lemma}

\newtheorem{remark}{Remark}

\newtheorem{theorem}[lemma]{Theorem}

\newtheorem{corollary}[lemma]{Corollary}

\newcommand{\EE}{{\mathbb{E}}}

\newcommand{\PP}{{\mathbb{P}}}

\def\rP{{\mathscr{P}}}

\def\rX{{\mathscr{X}}}
\def\rK{{P}}

\newcommand{\dZ}{\mathbb {Z}}

\newcommand{\dR}{\mathbb {R}}

\newcommand{\cW}{\mathcal {W}}

\newcommand{\cX}{\rX}

\newcommand{\tmix}{{\rm t}_{\textsc{mix}}}

\newcommand{\dd}{{\rm dist}}

\newcommand{\diam}{{\mathrm{diam}}}
\newcommand{\trel}{\mathrm{t}_{\textsc{rel}}}

\newcommand{\dkl}{\mathrm{d}_{\textsc{kl}}}
\newcommand{\vkl}{\mathscr{V}_{\textsc{kl}}}

\newcommand{\rL}{\mathscr{L}}

\title{Cutoff for non-negatively curved Markov chains}
\author{Justin Salez}
\begin{document}
\maketitle
\begin{abstract}
Discovered  by Aldous, Diaconis and Shahshahani in the context of card shuffling, the cutoff phenomenon has since then been established for a variety of Markov chains. However, proving cutoff remains a delicate affair, which requires a very detailed knowledge of the chain. Identifying the general mechanisms underlying this phase transition, without having to pinpoint its precise location, remains one of the most fundamental open problems in the area of mixing times. In the present paper, we make a step in this direction by establishing cutoff for all Markov chains with non-negative curvature, under a suitably refined product condition. The result applies, in particular,  to the random walk on abelian Cayley expanders satisfying a mild degree assumption, hence  to the random walk on \emph{almost all} abelian Cayley graphs. Our proof relies on a  quantitative \emph{entropic concentration principle}, which we believe to lie behind all cutoff phenomena.
\end{abstract}
\tableofcontents
\section{Introduction}
\subsection{The cutoff phenomenon}
\paragraph{Setup.} Throughout the paper, we consider a stochastic matrix  $\rK$  on a finite state space $\cX$,  and we let $(\rP_t)_{t\ge 0}$ denote the associated heat-kernel, defined for $t\in[0,\infty)$ and $x,y\in \cX$ by
\begin{eqnarray*}
\rP_t(x,y) & := & e^{-t}\sum_{k=0}^\infty \frac{\rK^k(x,y)t^k}{k!}.
\end{eqnarray*}
Any continuous-time Markov semi-group on a finite state space takes this form, after a trivial time scaling ensuring that jumps occur at rate at most $1$.  
As soon as $\rK$ is irreducible, we have  
\begin{eqnarray*}
\rP_t(x,y) & \xrightarrow[t\to\infty]{} & \pi(y),
\end{eqnarray*}
where $\pi=\pi P$ denotes the unique invariant law. The rate at which this convergence occurs is  captured by the so-called \emph{mixing time}, defined for any precision $\varepsilon\in(0,1)$ by 
\begin{eqnarray*}
\tmix(\varepsilon) & := & \min\left\{t\ge 0\colon \max_{x\in\rX}\|\rP_t(x,\cdot)-\pi\|_{\textsc{tv}}\le \varepsilon\right\},
\end{eqnarray*}
where $\|\mu-\pi\|_{\textsc{tv}}=\max_{A\subseteq \cX}|\mu(A)-\pi(A)|$ denotes total-variation distance. Understanding how this fundamental parameter depends on the  underlying transition matrix $P$ constitutes a fascinating area of research, at the crossroad between probability,  discrete geometry, spectral analysis and functional inequalities; see the books \cite{MR3726904,MR2341319} for an introduction. 

\paragraph{The cutoff phenomenon.}Quantifying the approach to equilibrium is of course  particularly relevant when the number of states is large. One is thus naturally led to consider a sequence of irreducible transition matrices  $(\rK_n)_{n\ge 1}$ and investigate the asymptotic behavior of their mixing times as $n\to\infty$. To lighten our exposition, we will completely drop the subscript $n$ from all notations. For many natural chains, a remarkable phase transition known as a \emph{cutoff}  has been established: the distance to equilibrium remains close to  its maximal value for a long time, and then suddenly drops to zero on a much shorter time-scale. More formally, for any fixed $\varepsilon\in(0,1)$, 
\begin{eqnarray}
\label{def:cutoff}
 \frac{\tmix(1-\varepsilon)}{\tmix(\varepsilon)} & = & 1+o(1),
\end{eqnarray}
where $o(1)$ denotes a quantity that vanishes as $n\to\infty$.
The first instances of this phenomenon were discovered  in the 80's by Aldous, Diaconis and Shahshahani in the context of card shuffling  \cite{diaconis1981generating,aldous1983mixing,aldous1986shuffling}. Since then, cutoff has been established on a variety of examples; see in particular the surveys \cite{diaconis1996cutoff,saloff2004random}. The present paper is concerned with the more fundamental problem of singling out abstract conditions under which this phase transition occurs, without having to pinpoint its precise location. This question was raised by Aldous and Diaconis in their seminal works \cite{aldous1983mixing,aldous1986shuffling},  and constitutes one of the most important open problems in the modern study of Markov chains.

\paragraph{The product condition.}
In the {reversible} case (i.e., $\pi(x)P(x,y)=\pi(y)P(y,x)$ for all $x,y\in\rX$),  cutoff  is easily seen to imply the so-called \emph{product condition}:  for each fixed $\varepsilon\in(0,1)$, 
\begin{eqnarray}
\label{def:product}
{\trel} & \ll & {\tmix(\varepsilon)},
\end{eqnarray}
 where the notation $a\ll b$ means $a/b=o(1)$, and where $\trel$ denotes the \emph{relaxation time} (inverse spectral gap) of the matrix $\rK$. The interest of this criterion is that it only involves  orders of magnitude: unlike the definition (\ref{def:cutoff}),  it can be checked without having to determine the precise prefactor in front of  mixing times. In the 2004 AIM workshop on mixing times, Peres \cite{peresamerican} conjectured that (\ref{def:product}) is also sufficient for cutoff. This has been verified for birth-and-death chains \cite{ding2010total} and, more generally, random walks on trees \cite{MR3650406}. Unfortunately, counter-examples have been constructed in the general case; see \cite[Section 6]{chen2008cutoff}. In fact, this failure is \emph{generic}, in the sense that (\ref{def:product}) is stable under a general perturbation which completely destroys cutoff; see  \cite[Example 18.7]{MR3726904}. Thus, the product condition will incorrectly predict cutoff for many natural Markov chains, including certain random walks on abelian groups. 
 
\paragraph{Other criteria.} A few other criteria for cutoff of reversible chains have  been proposed. In particular, Basu, Hermon and Peres provided a formulation in terms of concentration of hitting times \cite{MR3650406}. Unfortunately,  verifying this condition requires determining the precise prefactor in front of hitting times, which, to the best of our knowledge, has not been practically carried out beyond random walks on trees. In the specific context of random walks on regular graphs with fixed degree $d\ge 3$, the delicate \emph{asymptotic Ramanujan property}
\begin{eqnarray*}
\trel & = &  \frac{d}{d- 2\sqrt{d-1}}+o(1),
\end{eqnarray*}
is  known to imply cutoff \cite{2015arXivRamanujan,MR3693771,MR4178418}; see   \cite{bordenave2020cutoff} for a generalization. Note however, that  unlike the product condition (\ref{def:product}), verifying the Ramanujan property \emph{does} require determining the   relaxation time up to a $o(1)$ term, a notoriously challenging task even on random instances  \cite{MR2437174,MR4203039}. Finally, let us mention an impressive series of works by Lubetzky and Sly  \cite{lubetzky2014cutoff,MR3434254,MR3486171, MR3729612} (see also \cite{MR4144084}), which develops a general framework for proving cutoff in certain spin systems at sufficiently high temperature, without having to determine the cutoff location.

\paragraph{Our contribution.} For a broad class of chains known as \emph{non-negatively curved} chains, we provide a sufficient ``product-like'' condition for cutoff, which only involves comparing orders of magnitude of $\trel$ and $\tmix$. Moreover, we do not require reversibility, but only symmetry of the support:
\begin{eqnarray}
\label{assume:support}
\rK(x,y)>0 & \Longleftrightarrow & \rK(y,x)>0.
\end{eqnarray}
 Before stating the result, let us briefly recall the two notions of curvature on which our approach relies: the Bakry-\'Emery curvature, and the Ollivier-Ricci curvature.

\subsection{Non-negative curvature}
The \emph{Ricci curvature} is a far-reaching concept in Riemannian geometry; see e.g. the book \cite{MR3726907}. Over the past decades, considerable efforts have been made to develop a satisfactory analogue in discrete settings, such as graphs and Markov chains. In particular,  Bakry and \'Emery \cite{MR889476,MR941980} proposed a fruitful approach based on $\Gamma-$calculus; see also the book \cite{MR3155209}. More recently, Ollivier \cite{MR2484937} put forward a different definition of curvature based on optimal transport, which applies to arbitrary metric spaces. Our results will apply to both definitions indifferently.

\paragraph{Ollivier-Ricci curvature.} Assuming that the support of $\rK$ is symmetric, we may turn the state space $\cX$ into a metric space by equipping it with the distance
\begin{eqnarray*}
\dd(x,y) & := & \min\left\{k\in\dZ_+\colon \rK^k(x,y)>0\right\}.
\end{eqnarray*}
We use the graph-theoretical notation $x\sim y$ to mean that $\dd(x,y)=1$. 
As any distance, the above metric can be ``lifted'' to the distributional level via optimal transport. Specifically, the $L^1-$Wassertein distance between two probability measures $\mu$ and $\nu$ on $\cX$ is defined as
\begin{eqnarray*}
\cW_1(\mu,\nu) & := & \min_{\gamma\in\Pi(\mu,\nu)}\sum_{x,y\in \cX} \gamma(x,y)\dd(x,y),
\end{eqnarray*}
where $\Pi(\mu,\nu)$ is the set of all \emph{couplings} of $\mu$ and $\nu$, i.e. probability measures on $\cX\times \cX$ whose marginals are $\mu$ and $\nu$. The \emph{Ollivier-Ricci curvature} is the largest number $\kappa$ such that the inequality
\begin{eqnarray*}
\cW_1\left(\rP_t(x,\cdot),\rP_t(y,\cdot)\right) & \le & e^{-\kappa t}\,\dd(x,y),
\end{eqnarray*}
holds for all $t\in\dR_+$ and all $x,y\in \cX$. By Kantorovich duality, this is equivalent to the inequality
\begin{eqnarray}
\label{eq:contract}
\|\rP_tf\|_{\textsc{lip}} & \le & e^{-\kappa t}\|f\|_{\textsc{lip}},
\end{eqnarray}
for all $t\in\dR_+$ and all observables $f\colon\cX\to\dR$, where $\|f\|_{\textsc{lip}}=\max_{x\sim y} {|f(x)-f(y)|}$. In particular, non-negative Ollivier-Ricci curvature simply means that the semi-group is a contraction for the Lipschitz norm. This natural property constitutes the essence of the powerful \emph{path coupling method} \cite{MR2316551}, and its consequences in terms of geometry, mixing,  and concentration of measure  have been massively investigated. The literature is too vast for an exhaustive account, and we refer the reader to the  survey \cite{MR2648269} for details. Establishing non-negative curvature is easier than one may think. Indeed, by convexity of $\cW_1$, it is enough to prove the one-step estimate
\begin{eqnarray*}
\cW_1\left(\rK(x,\cdot),\rK(y,\cdot)\right) & \le & (1-\kappa)\,\dd(x,y).
\end{eqnarray*}
Furthermore, by the triangle inequality, we may restrict our attention to the case where $x\sim y$.

\paragraph{Bakry-\'Emery curvature.} We now turn to Bakry-\'Emery curvature. 
We only introduce the necessary notation, and refer to the beautiful lecture notes \cite{Ramon} for details. Write $\rL=\rK-\rm{Id}$ for the generator of the semi-group $(\rP_t)_{t\ge 0}$. This operator acts on any observable $f\colon \cX\to\dR$ via
\begin{eqnarray*}
(\rL f)(x) & := & \sum_{y\in\cX} \rK(x,y)\left(f(y)-f(x)\right),
\end{eqnarray*}
and the associated \emph{carré du champ} $\Gamma$ is given, for all $f,g\colon \cX\to\dR$, by the formula
\begin{eqnarray*}
\Gamma(f,g)(x) & := & \frac{1}{2}\sum_{y\in\cX}\rK(x,y)\left(f(y)-f(x)\right)\left(g(y)-g(x)\right). 
\end{eqnarray*}
The \emph{Bakry-\'Emery curvature} is then  defined as the largest number $\kappa$ such that the  inequality  
\begin{eqnarray*}
\frac{1}{2}\rL\Gamma(f,f) & \ge & \Gamma(f,\rL f)+\kappa\Gamma(f,f),
\end{eqnarray*}
holds pointwise, for all $f\colon\cX\to \dR$. This functional inequality, often denoted  $\textrm{CD}(\kappa,\infty)$, is equivalent to the following sub-commutativity relation between the carré du champ and the semi-group: 
\begin{eqnarray}
\label{eq:subco}
\forall t\in\dR_+,\quad \Gamma(\rP_t f,\rP_t f) & \le &  e^{-2\kappa t}  \rP_t\Gamma(f,f).
\end{eqnarray}
For consequences in terms of geometry, mixing,  and concentration of measure, see \cite{MR3782987}.
 
\paragraph{Non-negatively curved chains.}  In the discrete setting, there is apparently no known  relation  between the Bakry-\'Emery curvature and the Ollivier-Ricci curvature, although these two notions share many similarities. We emphasize that our results will apply to both definitions indifferently. Thus, by a \emph{non-negatively curved chain}, we will henceforth simply mean a chain that has non-negative curvature in the Bakry-\'Emery sense \emph{or} the Ollivier-Ricci sense. Non-negatively curved chains are ubiquitous, and appear in a broad variety of contexts. Classical examples include:
\begin{itemize}
\item all random walks on abelian groups \cite{MR3492631};
\item all conjugacy-invariant random walks on the permutation group  \cite{MR3646066,MR3936154}; 
\item simple random walks on all Bruhat graphs of Coxeter groups \cite{siconolfi2021ricci};
\item all monotone birth-and-death chains \cite{MR2348750};
\item the zero-range process with non-decreasing rates, and many other particle systems  \cite{MR4091580};
\item the Glauber dynamics for various spin systems at high-temperature \cite{MR3646066}.
\end{itemize}
To complement this list of examples, we note that non-negative curvature is closed under several natural operations, such as composition, superposition and $L^1-$tensorization \cite{MR2484937}.

\subsection{Results and implications}

Recall that $\rK$ is \emph{not} required to be reversible, i.e., to coincide with its adjoint 
\begin{eqnarray*}
\rK^\star(x,y) & := & \frac{\pi(y)P(y,x)}{\pi(x)}.
\end{eqnarray*}
Consequently, we define the relaxation time $\trel$ of the chain to be the inverse spectral gap of the additive reversibilization $\frac{\rK+\rK^\star}{2}$. Equivalently, $\trel$ is the optimal constant in the Poincaré inequality
\begin{eqnarray}
\label{def:poincare}
{\rm\bf Var}\left(f\right) & \le & \trel\,{\bf E}\left[\Gamma(f,f)\right],
\end{eqnarray}
valid for all observables $f\colon\rX\to\dR$, where $ {\bf E}$ and ${\rm \bf Var}$ denote expectation and variance on the finite probability space $(\rX,\pi)$. We also introduce the \emph{sparsity parameter}
\begin{eqnarray*}
\Delta & := & \max\left\{\frac{1}{P(x,y)}\colon x,y\in\rX,x\sim y\right\},
\end{eqnarray*}
which is simply the maximum degree when $P$ is the transition matrix of simple random walk on a graph. To avoid degeneracies, we will assume that $|\rX|\ge 3$, so that
\begin{eqnarray}
\label{assume:trivial}
\trel \ge \frac 12 & \textrm{ and } & \Delta\ge 2.   
\end{eqnarray}
Finally, we recall that the notation $a\ll b$  means that the ratio $a/b$ vanishes as  our implicit parameter $n$ tends to infinity. Similarly, $a \lesssim b$ means that $a/b$ is bounded above uniformly in $n$. We are now ready to state our main  result, in which \emph{non-negative curvature} can be understood either in the Ollivier-Ricci or the Bakry-\'Emery sense, indifferently. 

\begin{theorem}[Main result]\label{th:main} Consider a sequence of  irreducible transitions matrices with symmetric support and non-negative curvature. Suppose that for every fixed $\varepsilon\in(0,1)$, we have
\begin{eqnarray}
\label{assume:sparse}
\tmix(\varepsilon) & \gg & \left(\trel \log \Delta\right)^2.
\end{eqnarray} 
Then,  the sequence  exhibits cutoff. More precisely, for every $\varepsilon\in(0,\frac 12)$, we have
\begin{eqnarray*}
\tmix(\varepsilon)-\tmix(1-\varepsilon) & \lesssim & \sqrt{\tmix(1/4)}\,\trel\log \Delta.
\end{eqnarray*} 
\end{theorem}
Let us now comment on the ``product-like'' condition  (\ref{assume:sparse}). First, it is \emph{effective}, in the sense that its verification only requires comparing the orders of magnitude of $\tmix$ and $\trel$, as promised. Second, it implies the original product condition, by (\ref{assume:trivial}). Third, the presence of an additional sparsity term is unavoidable,  because of the generic counter-example in  \cite[Example 18.7]{MR3726904}. More precisely, let $P$ be the transition matrix  of the random walk with increment law $\mu$ on an abelian group $\rX$ (such chains are non-negatively curved), where the pair $(\rX,\mu)$ is chosen so that  $\trel\lesssim 1\ll \tmix(1/4)$. Then, we can destroy cutoff without affecting this property  by simply replacing $\mu$ with $(1-\theta)\mu+\theta\pi$, where $\theta\in(0,1)$ satisfies $\frac{1}{\tmix(1/4)}\ll \theta\ll 1$. However, this perturbation will drastically increase  the sparsity parameter $\Delta$, and the role of the latter in our condition (\ref{assume:sparse}) is precisely to preclude this type of pathologies. Finally, we emphasize that there is a variety of methods for estimating the orders of magnitude of $\tmix(\varepsilon)$ and $\trel$; see the books \cite{MR3726904,MR2341319}. For example, a simple diameter bound (see Lemma \ref{lm:diam} below) ensures that our ``product-like'' condition (\ref{assume:sparse}) holds whenever
\begin{eqnarray*}
\diam(\rX) & \gg & (\trel\log \Delta)^2.
\end{eqnarray*}
This condition may be further simplified by using the crude estimate $\diam(\rX)\ge \log N/\log\Delta$, where $N=|\rX|$ denotes the number of states. As a consequence, we readily deduce that non-negatively curved chains with reasonably good expansion exhibit cutoff, in the following sense.
\begin{corollary}[Non-negative curvature and expansion imply cutoff]\label{co:main}A sufficient condition for a sequence of non-negatively curved chains with symmetric support to exhibit cutoff is that
\begin{eqnarray}
\label{assume:expand}
 \trel & \ll & 
\frac{\left(\log N\right)^\frac{1}{2}}{\left(\log \Delta\right)^{\frac 32}}.
\end{eqnarray} 
\end{corollary}
This result applies to a variety of chains. To illustrate this, let us consider the important special case where $P$ is the transition matrix of simple random walk on a Cayley graph $G={\rm Cay}(\rX,S)$, where $(\rX,+)$ is a finite abelian group and $S\subseteq \rX$ a symmetric set of generators. Specifically,
\begin{eqnarray*}
P(x,y) & = & \frac{1}{|S|}\sum_{z\in S}{\bf 1}_{(y=x+z)}.
\end{eqnarray*}
Write  $N=|\rX|$ for the number of group elements,  and  $d=|S|$ for the number of generators.
\begin{corollary}[Abelian graphs]\label{co:cayley} Random walk on abelian Cayley graphs  exhibits cutoff whenever
\begin{eqnarray*}
\trel & \ll & \frac{\left(\log N\right)^{1/2}}{\left(\log d\right)^{3/2}}.
\end{eqnarray*}
\end{corollary}
In particular, this applies to random instances. Indeed, a celebrated result of Alon \& Roichman \cite{MR1262979}, refined by  Pak \cite{MR1729149}, by Naor \cite{MR2942733}, and finally by Hermon \& Olesker-Taylor  \cite{hermon2021cutoff}, asserts that \emph{almost all Cayley graphs with  $d\ge (1+\varepsilon)\log_2 N$ satisfy $\trel\lesssim 1$}, leading to the following result.

\begin{corollary}[Cutoff on almost all abelian Cayley graphs]Let $G$ be the Cayley graph  obtained by choosing $d$ generators  uniformly at random in an abelian group of size $N$. Consider the regime
\begin{eqnarray*}
d \ge  a\log_2 N & \textrm{ and } & \log d \ll \left(\log N\right)^{\frac 13},
\end{eqnarray*}
where $a>1$ is any fixed constant. Then, the random walk on  $G$ exhibits cutoff with high probability.
\end{corollary}
Note that the requirement $a>1$ can not be improved in general, since the binary group $\dZ_2^d$ can not be generated by less than $d=\log_2N$ elements. 
The problem of establishing cutoff for random abelian Cayley graphs has a long history (see the  survey \cite{MR2121795}). It originates  with a conjecture raised by Aldous and Diaconis in an extended version of \cite{aldous1986shuffling}. The dense regime $d\gg\log N$ was settled several years ago by Dou and Hildebrand \cite{MR1404540,MR1296428}. The sparse regime $d\lesssim \log N$, in contrast,  was tackled only very recently in an impressive series of works by Hermon and Olesker-Taylor  \cite{hermon2021cutoff,nonabelian,hermon2021results,hermon2021geometry}, which additionally provides a very detailed picture of the geometry of random Cayley graphs. We emphasize that those two approaches crucially rely on the fact that the generators are chosen uniformly at random, whereas  our Corollary \ref{co:cayley} deterministically applies to \emph{any} abelian Cayley graph with reasonably good expansion, without requiring any specific computation. 

\begin{remark}[Refinements]\label{rk:refinement1} With some additional knowledge on the chain, our approach can easily be  refined. For example, the ``product-like'' condition (\ref{assume:sparse}) can be replaced with 
\begin{eqnarray*}
\tmix(\varepsilon) & \gg & \frac{\trel\log \Delta}{\sqrt{\kappa}},
\end{eqnarray*}
which is strictly weaker  as soon as  $ 1/\kappa\ll \tmix(\varepsilon)$. See Remark \ref{rk:refinement2} below for details.
\end{remark}

\begin{remark}[Characterization of the cutoff time]\label{rk:location}Our proof also provides an entropic characterization of the cutoff time; see the second part of Theorem \ref{th:cutoff} below.
\end{remark}

\begin{remark}[Bounded degrees] The condition (\ref{assume:sparse}) trivially holds in the bounded-degree expander regime $\max(\Delta,\trel)\lesssim 1\ll |\rX|$. Unfortunately, such chains must have negative curvature  \cite{salez2021sparse}. 
\end{remark}

\paragraph{Proof outline.}
The proof of Theorem \ref{th:main} relies on a certain \emph{entropic concentration phenomenon}, which we formulate in Section \ref{sec:entropy} below. In Section \ref{sec:proof}, we show that this phenomenon implies cutoff along \emph{any} sequence of Markov chains, without any curvature or support assumption. This is inspired by a recent body of works establishing cutoff on random instances \cite{MR2667423,MR3650414,MR3758735,MR3773804,MR3916108,MR4132638,
conchonkerjan2019cutoff,bordenave2020cutoff,hermon2020universality,hermon2021cutoff}, where entropy plays a crucial role (see also \cite{MR4021248,MR4178418}). Our  entropic criterion can be understood as the common  mechanism underlying these generic cutoff phenomena, and hopefully many more to come. Finally, in Section \ref{sec:varentropy}, we establish the entropic concentration phenomenon for non-negatively curved chains satisfying  our ``product-like'' condition. This combines a new gradient estimate for the logarithm of the heat-kernel with a {local} concentration inequality for Lipschitz observables. While connections between curvature and local concentration are well known (see e.g., \cite{MR3726607,MR2683634,Ramon} and the references therein), their application to cutoff seems new and promising.

\paragraph{Acknowledgement.} The author warmly thanks Nathanaël Berestycki, Max Fathi, Jonathan Hermon and Sam Olesker-Taylor  for their valuable comments on a first version of the paper. This work was partially supported by Institut Universitaire de France.
\section{Proof}
\subsection{The entropic concentration phenomenon}
\label{sec:entropy}
\emph{Relative entropy} (or \emph{Kullback-Leibler divergence}) is a natural  measure of  discrepancy  between a probability measure $\mu$ and a (fully supported) reference probability measure $\pi$. It is defined as 
\begin{eqnarray*}
\dkl\left(\mu\|\pi\right) & := & \sum_{x\in \rX}\mu(x)\log\frac{\mu(x)}{\pi(x)},
\end{eqnarray*}
where $\log$ denotes the natural logarithm. 
By the strict convexity of $u\mapsto u\log u$, we always have $\dkl\left(\mu\|\pi\right)\ge 0$, with equality if and only if $\mu=\pi$. The celebrated Pinsker Inequality provides a one-sided quantitative version of this statement, by ensuring that $\|\mu-\pi\|_{\textsc{tv}}$ is small whenever $\dkl\left(\mu\|\pi\right)$ is small. This is the starting point of a variety of powerful upper bounds on mixing times (see, e.g., \cite{MR2283379}). However, the converse relation -- namely, that $\|\mu-\pi\|_{\textsc{tv}}$  has to be large whenever $\dkl\left(\mu\|\pi\right)$ is large --  is much looser, because a very small region with $\frac{\mu}{\pi}$  large can boost the above sum, while being negligible from a total-variation viewpoint. A simple way to preclude this type of pathologies is to have some control on the typical fluctuations of $\frac{\mu}{\pi}$. To do so, it is natural to consider the associated variance, beautifully called \emph{varentropy} in a different context \cite{MR4073681}:
\begin{eqnarray*}
\vkl\left(\mu\|\pi\right) & := & \sum_{x\in \rX}\mu(x)\left(\log\frac{\mu(x)}{\pi(x)}-\dkl\left(\mu\|\pi\right)\right)^2.
\end{eqnarray*}
Coming back to Markov chains, we will consider the worst-case varentropy under the heat-kernel:
\begin{eqnarray*}
\vkl^\star(t) & := & {\max_{o\in\rX}\vkl\left(\rP_t(o,\cdot)\|\pi\right)}.
\end{eqnarray*}
This key quantity turns out to govern the cutoff phenomenon, as formalized in the following  result (proved in Section \ref{sec:proof} below). We emphasize that the latter applies to \emph{any} transition matrix $P$: neither non-negative curvature, nor symmetry of the support is  required. To the best of our knowledge, the use of varentropy in relation with  the cutoff phenomenon is new.
\begin{theorem}[Entropic concentration implies cutoff]\label{th:cutoff}For any stochastic matrix $P$ and any $\varepsilon\in(0,\frac{1}{2})$
\begin{eqnarray*}
\tmix(\varepsilon)-\tmix(1-\varepsilon) & \le & \frac{2\trel}{\varepsilon^2}\left[1+\sqrt{\vkl^\star\left(\tmix\left(1-\varepsilon\right)\right)}\right].
\end{eqnarray*}
In particular, for a sequence of stochastic matrices to exhibit cutoff, it is enough that it satisfies 
\begin{eqnarray}
\label{def:entropic}
1+\sqrt{\vkl^\star\left(\tmix\left(\varepsilon\right)\right)} & \ll & \frac{\tmix(\varepsilon)}{\trel},
\end{eqnarray}
for all $\varepsilon\in\left(\frac 12,1\right)$. Moreover, in that case, we have $\tmix(\varepsilon)\sim t$, where $t$ solves the equation
\begin{eqnarray*}
\dkl^\star(t) & \asymp & 1+\sqrt{\vkl^\star\left(t\right)},
\end{eqnarray*}
with $\dkl^\star(t)=\max_{o\in\rX}\dkl(\rP_t(o,\cdot)\|\pi)$, and with $\asymp$ denoting equality up to any fixed prefactor.
\end{theorem}
 We naturally  call (\ref{def:entropic}) the \emph{entropic concentration phenomenon}. Observe that the latter readily implies the product condition (\ref{def:product}). Of course, to make our criterion effective, we need to complement it with an estimate on the varentropy $\vkl^\star$. This is precisely the aim of our second key result, established in Section \ref{sec:varentropy} below, and which crucially exploits non-negative curvature.
\begin{theorem}[Varentropy estimate]\label{th:varentropy}Consider a sequence of   non-negatively curved transitions matrices with symmetric support. Fix $\varepsilon\in(0,1)$, and suppose that 
$
\sqrt{\trel}\ll  \tmix(\varepsilon).
$
Then, 
\begin{eqnarray*}
\vkl^\star(\tmix(\varepsilon)) & \lesssim & \tmix(\varepsilon)\left(\log\Delta\right)^2.
\end{eqnarray*}
\end{theorem}
When combined together,  Theorems \ref{th:cutoff} and \ref{th:varentropy}  readily imply Theorem \ref{th:main}. The remainder of the paper is  thus devoted to the proof of these two results.

\subsection{Entropic concentration implies cutoff}
\label{sec:proof}
In this section, we prove that entropic concentration implies cutoff, as stated in Theorem \ref{th:cutoff}. To do so, we need a sharp, two-sided quantitative relation between entropy and mixing. We start with the following upper bound, which shows that mixing occurs quickly once relative entropy is small.

\begin{lemma}[Entropic upper-bound]\label{lm:upper}For all $t\in\dR_+$ and $\varepsilon\in(0,1)$,
 \begin{eqnarray*}
\tmix(\varepsilon) & \le & t+ \frac{\trel}{\varepsilon}\left( 1+\dkl^\star\left(t\right)\right).
\end{eqnarray*}
\end{lemma}
\begin{proof}Recall that the relaxation time can be used to bound the total-variation distance to equilibrium via the following classical inequality: for any law $\mu$ on $\rX$  and any time $s\in\dR_+$,
\begin{eqnarray}
\label{gap:contraction}
\left\|\mu \rP_s-\pi\right\|_{\textsc{tv}} & \le & \frac{e^{-s/\trel\, }}{2} \sqrt{\left\|\frac{\mu}{\pi}\right\|_\infty};
\end{eqnarray}
see, e.g., \cite{MR2341319}. 
Now fix a law $\mu$ on $\rX$, and consider the set $A\subseteq \rX$ defined by 
\begin{eqnarray*}
A& := & \left\{x\in\rX\colon \log\frac{\mu(x)}{\pi(x)} < 1+\frac{2\dkl\left(\mu\|\pi\right)}{\varepsilon}\right\}.
\end{eqnarray*}
Observe that by definition,
\begin{eqnarray*}
\left(1+\frac{2\dkl\left(\mu\|\pi\right)}{\varepsilon}\right) \mu(A^c) & \le & \sum_{x\in A^c}\mu(x)\log\frac{\mu(x)}{\pi(x)} \\& = & \dkl\left(\mu\|\pi\right)+\sum_{x\in A}\mu(x)\log\frac{\pi(x)}{\mu(x)} \\ & \le &  \dkl\left(\mu\|\pi\right)+\pi(A)-\mu(A)\\
& \le  & \dkl\left(\mu\|\pi\right) +\mu(A^c)
\end{eqnarray*}
where at the third line we have used $\log u\le u-1$. After simplification, we are left with 
\begin{eqnarray*}
\mu(A^c) & \le & \frac{\varepsilon}{2}.
\end{eqnarray*}
Now, let $\widehat{\mu}:=\mu(\cdot|A)$ be $\mu$ conditioned on $A$. Note that
\begin{eqnarray*}
\left\|\frac{\widehat{\mu}}{\pi}\right\|_\infty& 
 = & \frac{1}{\mu(A)}\max_{x\in A}\frac{\mu(x)}{\pi(x)} \ \le \ \exp\left\{2+\frac{2\dkl\left(\mu\|\pi\right)}{\varepsilon}\right\},
\end{eqnarray*}
because $\mu(A)\ge 1/2 \ge 1/e$. Consequently, (\ref{gap:contraction}) applied to $\widehat\mu$ yields
\begin{eqnarray*}
\left\|\widehat{\mu}\rP_s-\pi\right\|_{\textsc{tv}} & \le & \frac{1}{2}\exp\left\{1+\frac{\dkl\left(\mu\|\pi\right)}{\varepsilon}-\frac{s}{\trel}\right\},
\end{eqnarray*}
for all $s\ge 0$. To make the right-hand side less than $\varepsilon/2$, we choose
\begin{eqnarray}
\label{def:s}
s & := & \frac{\trel}{\varepsilon}\left({1+\dkl(\mu\|\pi)}\right).
\end{eqnarray}
On the other hand, we trivially have
\begin{eqnarray*}
\|\widehat{\mu} \rP_s-\mu \rP_s\|_{\textsc{tv}} & \le & \|\widehat{\mu}-\mu \|_{\textsc{tv}}  \ = \ \mu(A^c) \ \le \ \frac{\varepsilon}{2}.
\end{eqnarray*}
By the triangle inequality, we conclude that for $s$ as in (\ref{def:s}),  
\begin{eqnarray*}
\|{\mu} \rP_s-\pi\|_{\textsc{tv}} & \le & \varepsilon.
\end{eqnarray*}
Since $\mu$ is arbitrary, we may take $\mu=\rP_t(o,\cdot)$, and then maximize over  $o\in\rX$ to conclude. 
\end{proof}
To complement the above estimate, we now provide a lower bound showing that mixing can not occur until the relative entropy has reached a sufficiently low level.
\begin{lemma}[Entropic lower-bound]\label{lm:lower}For any probability measure $\mu$ on $\rX$ and any $\varepsilon\in(0,1)$, 
\begin{eqnarray*}
 \|\mu-\pi\|_{\textsc{tv}} \le 1-\varepsilon & \Longrightarrow &
\dkl\left(\mu\|\pi\right)   \le  \frac{1+\sqrt{\vkl\left(\mu\|\pi\right)}}{\varepsilon}  .
\end{eqnarray*}
\end{lemma}
\begin{proof}
Consider the event $A\subseteq\rX$ defined by
\begin{eqnarray*}
A   :=   \left\{x\in \rX\colon  \mu(x) \ge \pi(x)e^\theta  \right\},& \textrm{ with } & \theta = \dkl(\mu\|\pi)-\frac{\sqrt{\vkl(\mu\|\pi)}}{\varepsilon}.
\end{eqnarray*}
Since  $\log \frac{\mu}{\pi}$ has mean $\dkl(\mu\|\pi)$ and variance $\vkl(\mu\|\pi)$ under $\mu$, Chebychev's inequality implies
\begin{eqnarray*}
\mu(A) & \ge &  1-\varepsilon^2.
\end{eqnarray*}
On the other hand, the definition of $A$ readily implies 
\begin{eqnarray*}
\pi(A) & \le & e^{-\theta}\mu(A).
\end{eqnarray*}
Together, these two inequalities imply 
\begin{eqnarray*}
\mu(A)-\pi(A) & \ge & (1-\varepsilon^2)(1-e^{-\theta}).
\end{eqnarray*}
Assuming that $\|\mu-\pi\|_{\textsc{tv}}\le 1-\varepsilon$, we deduce that $1-\varepsilon\ge (1-\varepsilon^2)(1-e^{-\theta})$, or equivalently,
\begin{eqnarray*}
\theta & \le & \log\left(1+\frac{1}{\varepsilon}\right).
\end{eqnarray*}
Since $\log(1+u)\le u$, this implies $\theta\le 1/\varepsilon$, as desired.  
\end{proof}
With these lemmas at hand, the proof of Theorem \ref{th:cutoff} is now straightforward.
\begin{proof}[Proof of Theorem \ref{th:cutoff}]Fix  $\varepsilon\in(0,1)$ and set $t=\tmix(1-\varepsilon)$. By Lemma \ref{lm:upper}, we have
 \begin{eqnarray*}
\tmix(\varepsilon)  & \le & t+ \frac{\trel}{\varepsilon}\left( 1+\dkl^\star\left(t\right)\right).
\end{eqnarray*}
On the other hand, Lemma \ref{lm:lower} with $\mu=\rP_t(o,\cdot)$ (followed by a maximization over $o)$ forces 
\begin{eqnarray*}
\dkl^\star\left(t\right) & \le & \frac{1+\sqrt{\vkl^\star(t)}}{\varepsilon}.
\end{eqnarray*}
Reinserting this above and using $\varepsilon\le 1$ immediately yields the desired claim.
\end{proof}

\subsection{Non-negative curvature implies entropic concentration}
\label{sec:varentropy}
In this section, we prove the general varentropy estimate for non-negatively curved chains stated in Theorem \ref{th:varentropy}.  Our starting point is the following local concentration inequality for Lipschitz observables. The term \emph{local} here refers to the fact that the underlying measure is the heat-kernel itself, rather than the equilibrium measure $\pi$. Let $\kappa$ denote the curvature of the chain, in either the Ollivier-Ricci or the Bakry-\'Emery sense. This is the only place where  curvature is used.  
\begin{lemma}[Local concentration inequality]\label{lm:concentration}For any  $f\colon \rX\to \dR$ and $t\in\dR_+$,
\begin{eqnarray*}
\rP_t(f^2)-\left(\rP_t f\right)^2 & \le & \frac{1-e^{-2t\kappa  }}{\kappa }\|f\|_{\textsc{lip}}^2,
\end{eqnarray*}
where the fraction is interpreted as $2t$ if $\kappa=0$. In particular, if $\kappa\ge 0$, then 
\begin{eqnarray*}
\rP_t(f^2)-\left(\rP_t f\right)^2 & \le & 2t\|f\|_{\textsc{lip}}^2.
\end{eqnarray*}
\end{lemma}
\begin{proof}
Our starting point is the following well-known identity, which is easily checked by differentiating both sides with respect to $t$ (see, e.g. \cite[Problem 2.12.a]{Ramon}):
\begin{eqnarray*}
\rP_t(f^2)-\left(\rP_t f\right)^2 & = & 2\int_0^t\rP_{t-s}\Gamma\left(\rP_sf,\rP_sf\right){\rm d}s.
\end{eqnarray*}
This reduces our task to proving that  the integrand is at most $e^{-2\kappa s}\|f\|_{\textsc{lip}}^2$. In the Bakry-\'Emery case, we can use the sub-commutativity property (\ref{eq:subco}) to write
\begin{eqnarray*}
\Gamma\left(\rP_sf,\rP_sf\right)  & \le & e^{-2\kappa s}\rP_s\Gamma\left(f,f\right)\\
& \le & e^{-2\kappa s}\|f\|_{\textsc{lip}}^2,
\end{eqnarray*}
where the second line follows from the trivial bound $\Gamma(f,f)\le \|f\|_{\textsc{lip}}^2$ and the stochasticity of the operator $\rP_s$. On the other hand, in the Ollivier-Ricci case, we use  (\ref{eq:contract}) to write
\begin{eqnarray*}
\Gamma\left(\rP_sf,\rP_sf\right) 
& \le & \|\rP_sf\|_{\textsc{lip}}^2\\
& \le & e^{-2\kappa s} \|f\|_{\textsc{lip}}^2.
\end{eqnarray*}
In either case, we obtain $\Gamma\left(\rP_sf,\rP_sf\right)\le e^{-2\kappa s} \|f\|_{\textsc{lip}}^2$. Since this uniform bound is trivially preserved under the action of the stochastic operator $\rP_{t-s}$, the claim is established.
\end{proof}
\begin{remark}[Refinement]\label{rk:refinement2}The second part of the lemma uses the crude bound
\begin{eqnarray*}
\frac{1-e^{-2t\kappa  }}{\kappa } & \le & 2t,
\end{eqnarray*}
which has the advantage of suppressing the dependency in the curvature. However, in situations where a quantitative lower-bound on $\kappa$ is known, the alternative bound 
\begin{eqnarray*}
\frac{1-e^{-2t\kappa  }}{\kappa } & \le & \frac{1}{\kappa},
\end{eqnarray*}
might be preferable, and  leads to the refined condition $\tmix(\varepsilon)\gg \frac{\trel\log\Delta}{\sqrt{\kappa}}$ mentioned in Remark \ref{rk:refinement1}.
\end{remark}
Applying Lemma \ref{lm:concentration} to the observable $f(x)=\log\frac{\rP_t(o,x)}{\pi(x)}$ readily yields the varentropy estimate
\begin{eqnarray}
\label{key:k}
\forall t\ge 0,\quad \vkl^\star(t) & \le & 2t\max_{o\in\rX}\left\|\log \frac{\rP_t(o,\cdot)}{\pi(\cdot)}\right\|_{\textsc{lip}}^2.
\end{eqnarray}
This reduces our task to obtaining a gradient estimate on the logarithm of the heat-kernel. While such estimates have been explored in the setting of diffusions on manifolds (see, e.g., \cite{MR1885762}), we could not find any satisfactory  analogue on discrete state spaces. Here is what we can prove.
\begin{lemma}[Logarithmic gradient estimate]\label{lm:lip}If $P$ has symmetric support, then 
\begin{eqnarray*}
\left\|\log \frac{\rP_t(o,\cdot)}{\pi(\cdot)}\right\|_{\textsc{lip}}& \le & 3\left(1+\log\Delta\right),
\end{eqnarray*}
for any initial state $o\in\rX$ and any time $t\ge\diam(\rX)/4$.
\end{lemma}
\begin{proof}
Fix an initial state $o\in \rX$ and a time $t> 0$, and define $f\colon\rX\to(0,\infty)$ by
\begin{eqnarray*}
f(x) & := & \frac{\rP_t(o,x)}{\pi(x)}.
\end{eqnarray*}
Let also $q$ denote the Poisson distribution with mean $t$, i.e.
\begin{eqnarray*}
\forall k\in\dZ_+,\quad q(k) & = & \frac{t^ke^{-t}}{k!}.
\end{eqnarray*}
It follows from the definitions of $P^\star$, $f$, and $\rP_t$ that for all $x,y\in\rX$,
\begin{eqnarray*}
\pi(x)P^\star(x,y)f(y) &  = &  P(y,x)\rP_t(o,y) \ = \  \sum_{k=0}^\infty q(k)P^k (o,y)P(y,x).
\end{eqnarray*}
Summing over all $y\in\rX$ and using the Poisson identity $tq(k)=(k+1)q(k+1)$, we arrive at
\begin{eqnarray*}
t\pi(x)(P^\star f)(x) & = & \sum_{k=0}^\infty k q(k)P^{k}(o,x).
\end{eqnarray*}
We may now divide both sides by $\rP_t(o,x)=\sum_{k}q(k)P^k(o,x)$ to obtain
\begin{eqnarray*}
\frac{t(P^\star f)(x)}{f(x)} & = & \frac{\sum_{k=0}^\infty k q(k)P^{k}(o,x)}{\sum_{k=0}^\infty  q(k)P^{k}(o,x)}.
\end{eqnarray*}
By Jensen's inequality, we deduce from this expression that  
\begin{eqnarray*}
\frac{t(P^\star f)(x)}{f(x)} & \le & \log\left(\frac{\sum_{k=0}^\infty e^{ k} q(k)P^{k}(o,x)}{\sum_{k=0}^\infty  q(k)P^{k}(o,x)}\right)\\
& \le & \log\left(\frac{e^{t(e-1)}}{\rP_t(o,x)}\right),
\end{eqnarray*}
where the second inequality  simply uses the bound $P^{k}(o,x)\le 1$ in the numerator. In other words, 
\begin{eqnarray*}
\sum_{y\in\rX}P^\star(x,y)\frac{f(y)}{f(x)} & \le & e-1+\frac{1}{t}\log\frac{1}{\rP_t(o,x)}.
\end{eqnarray*}
Using the notation $(\Delta(Q))^{-1}$ for the smallest non-zero entry of a matrix $Q$, this readily implies  
\begin{eqnarray*}
\max_{y\sim x}\frac{f(y)}{f(x)} & \le & \Delta(P^\star)\left(e-1+\frac{\log\Delta(\rP_t)}{t}\right).
\end{eqnarray*}
Taking logarithms, we obtain the logarithmic gradient estimate
\begin{eqnarray}
\label{key}
\|\log f\|_{\textsc{lip}} & \le & \log\Delta(P^\star)+\log \left(e-1+\frac{\log\Delta(\rP_t)}{t}\right).
\end{eqnarray}
It now only  remains to bound $\Delta(P^\star)$ and $\Delta(\rP_t)$. For the former, we simply remark that 
\begin{eqnarray}
\label{Delta1}
\Delta(P^\star) & \le & \Delta^2(P)=\Delta^2,
\end{eqnarray}
as is easily deduced from the identity  $P^\star(x,y)P^\star(y,x)  =  P(x,y)P(y,x)$ and the symmetry of the support of $P$. To estimate $\Delta(\rP_t)$,  we consider the $\frac 34-$idle transition matrix $\widehat{P}=\frac{3}{4}{\rm Id}+\frac{1}{4}P$. Note that for $k=\diam(\rX)$, all entries of $\widehat{P}^k$ are at least $\left(\frac 1{4\Delta}\right)^{k}$. Consequently, for every $x,y\in\rX$,  
\begin{eqnarray*}
k\ge \diam(\rX) & \Longrightarrow &  \widehat{P}^k\left(x,y\right)   \ge  \left(\frac{1}{4\Delta}\right)^{\diam(\rX)}.
\end{eqnarray*}
Multiplying by $e^{-4t}(4t)^k/k!$ and summing over all $k\ge\diam(\rX)$, we obtain
\begin{eqnarray*}
\rP_t(x,y) & \ge & p\left(\frac{1}{4\Delta}\right)^{\diam(\rX)},
\end{eqnarray*}
where $p$ denotes the probability that a Poisson variable with mean $4t$ is at least $\diam(\rX)$. Choosing $t\ge \frac{\diam(\rX)}{4}$ makes this probability at least $1/2$, and we deduce that
\begin{eqnarray}
\label{Delta2}
\Delta(\rP_t) & \le & 2(4\Delta)^{\diam(\rX)}.
\end{eqnarray}
Inserting the estimates (\ref{Delta1})-(\ref{Delta2}) into (\ref{key}) easily yields the claim.
\end{proof}

Our last ingredient is the following elementary diameter bound, in which the fact that $\varepsilon$ may be taken arbitrarily close to $1$ is crucial.  
\begin{lemma}[Diameter bound]\label{lm:diam}For any $\varepsilon\in(0,1)$, we have
\begin{eqnarray*}
\diam(\rX) & \le & 2\tmix(\varepsilon)+\sqrt{\frac{8\tmix(\varepsilon)}{1-\varepsilon}}+\sqrt{\frac{8\trel}{1-\varepsilon}}.
\end{eqnarray*}
\end{lemma}
\begin{proof}
Fix $\varepsilon\in(0,1)$, and set $t=\tmix(\varepsilon)$. By definition, we have
\begin{eqnarray}
\label{pip}
 \rP_t(o,A) & \le & \pi(A)+\varepsilon,
 \end{eqnarray} 
for any initial state $o\in\rX$ and any event $A\subseteq\rX$. Let us consider the specific choice 
\begin{eqnarray*}
A & := & \left\{x\in\rX\colon \dd(o,x)\le t+\sqrt{\frac{2t}{1-\varepsilon}}\right\}.
\end{eqnarray*}
Note that by Chebychev's inequality, 
\begin{eqnarray*}
\rP_t(o,A) &  > & \frac{1+\varepsilon}{2}, 
\end{eqnarray*}
because  the distance to the origin at time $t$ is stochastically dominated by a Poisson random variable with mean $t$. In view of (\ref{pip}), we deduce that $\pi(A)> (1-\varepsilon)/2$, i.e. 
\begin{eqnarray}
\label{vareps}
\PP\left(\dd(o,U)  \le t+\sqrt{\frac{2t}{1-\varepsilon}}\right) & > & \frac{1-\varepsilon}{2}, 
\end{eqnarray}
where $U$ here denotes a random variable with distribution $\pi$. 
On the other hand, the function $f\colon x\mapsto \dd(o,x)$ is trivially Lipschitz on $\rX$, so the Poincar\'e inequality (\ref{def:poincare}) implies
\begin{eqnarray*}
{\rm Var}\left(\dd(o,U)\right) & \le & \trel. 
\end{eqnarray*}
By Chebychev's inequality again, we deduce that
\begin{eqnarray}
\label{vareps2}
\PP\left(\dd(o,U)-\EE[\dd(o,U)]\in\left[-\sqrt{\frac{2\trel}{1-\varepsilon}},\sqrt{\frac{2\trel}{1-\varepsilon}}\right]\right) & \ge &  \frac{1+\varepsilon}{2}.
\end{eqnarray}
Thus, the events in (\ref{vareps})-(\ref{vareps2}) must intersect. In other words, for all $o\in\rX$,
\begin{eqnarray*}
\EE\left[\dd(o,U)\right] & \le & t+\sqrt{\frac{2t}{1-\varepsilon}}+\sqrt{\frac{2\trel}{1-\varepsilon}}.
\end{eqnarray*}
The triangle inequality $\dd(o,o')\le \EE[\dd(o,U)+\dd(o',U)]$ completes the proof.
\end{proof}
We may at last establish Theorem \ref{th:varentropy}.
\begin{proof}[Proof of Theorem \ref{th:varentropy}]
Fix $\varepsilon\in(0,1)$ and suppose that $\sqrt{\trel}\ll {\tmix(\varepsilon)}$. 
By, Lemma \ref{lm:diam}, we have
\begin{eqnarray*}
\diam(\rX) & \le & (2+o(1))\tmix(\varepsilon).
\end{eqnarray*}
In particular, the condition $\tmix(\varepsilon)\ge\diam(\rX)/4$ is eventually satisfied. Consequently, Lemma \ref{lm:lip} applies with $t=\tmix(\varepsilon)$. Combining this with the inequality (\ref{key:k}), we obtain
\begin{eqnarray*}
\vkl^\star(\tmix(\varepsilon)) & \le & 18\tmix(\varepsilon)\left(1+\log \Delta\right)^2,
\end{eqnarray*}
for $n$ sufficiently large. Since  $\Delta \ge 2$, we have $18(1+\log \Delta)^2\lesssim\log^2\Delta$, as desired.
\end{proof}
\bibliographystyle{plain}
\bibliography{draft}
\end{document}